\documentclass[12pt,a4paper,reqno]{amsart}
\usepackage[hmargin=2.5cm,bmargin=2.5cm,tmargin=3cm]{geometry}
\usepackage{enumerate}
\usepackage{amsmath,amsthm,amssymb,amsfonts,latexsym}
\usepackage[gen]{eurosym}
\usepackage{orcidlink}
\usepackage[normalem]{ulem}
\usepackage[german,english]{babel}
\usepackage{graphicx}
\usepackage{tikz}
\usepackage{xfrac}
\usetikzlibrary{shapes,arrows,decorations.pathmorphing,backgrounds,fit,positioning,shapes.symbols,chains,graphs,graphs.standard}
\usepackage{dsfont}
\usepackage{upgreek}
\usepackage{tikz-cd}
\usepackage{caption}
\usepackage{faktor}
\allowdisplaybreaks
\usepackage{latexsym}

\usepackage{tikz}					

\newcommand{\Tr}{{\mathop{\mathrm{Tr} \,}}}

\newcommand{\RR}{\mathbb{R}}

\newtheorem{theorem}{Theorem}
\newtheorem{lemma}[theorem]{Lemma}
\newtheorem{definition}[theorem]{Definition}
\newtheorem{proposition}[theorem]{Proposition}

\newtheorem{remark}[theorem]{Remark}

\newtheorem{example}[theorem]{Example}
\newcommand{\be}{\begin{equation}}
\newcommand{\ee}{\end{equation}}
\newcommand{\bea}{\begin{eqnarray*}}
	\newcommand{\eea}{\end{eqnarray*}}
\newcommand{\beq}{\begin{eqnarray}}
\newcommand{\eeq}{\end{eqnarray}}

\newtheorem{prop}[theorem]{Proposition}

\usepackage{todonotes}

\title[Surface area of graphs, connectivity measures and spectral estimates]{On the surface area of graphs, related connectivity measures and spectral estimates} 

\subjclass[2010]{}


\author[P.~Bifulco]{Patrizio Bifulco\orcidlink{0009-0004-0628-374X}}
\author[J.~Kerner]{Joachim Kerner\orcidlink{0000-0003-0638-4183}}

\address{Lehrgebiet Analysis, Fakult\"at Mathematik und Informatik, Fern\-Universit\"at in Hagen, D-58084 Hagen, Germany}
\email{patrizio.bifulco@fernuni-hagen.de}

\address{Lehrgebiet Angewandte Stochastik, Fakult\"at Mathematik und Informatik, Fern\-Universit\"at in Hagen, D-58084 Hagen, Germany}
\email{joachim.kerner@fernuni-hagen.de}

\date{\today}

\thanks{
}

\begin{document}

	\begin{abstract} In this note we elaborate on some notions of surface area for discrete graphs which are closely related to the inverse degree. These notions then naturally lead to associated connectivity measures of graphs and to the definition of a special class of large graphs, called social graphs, that might prove interesting for applications. In addition, we derive spectral estimates involving the surface area and, as a main result, present an upper bound on the second eigenvalue for planar graphs which in some cases improves upon existing bounds from \cite{SpTe} and \cite{MP21}.
	\end{abstract}
	
\maketitle

	\section{Introduction}
	
	Discrete graphs have quite a long and rich history in mathematics with many and important applications in other sciences ranging from chemistry to computer science. A related important concept is that of metric graphs since they provide a fruitful link between one-dimensional and higher-dimensional aspects. In the realm of mathematical physics, metric graphs often go under the name of quantum graphs whenever one studies differential operators on such graphs that are motivated by quantum mechanics~\cite{BerkolaikoBook,KurasovBook}. A classical research direction on graphs tries to understand the connection between geometrical properties of the graph and spectral properties of operators defined on them~\cite{Chung}. For instance, a well-known operator on graphs is the so-called Laplacian which is a positive operator that has, on a finite discrete graph, as many eigenvalues as there are vertices (counting the eigenvalues with multiplicity). The first eigenvalue is always equal to zero and its multiplicity equals the number of connected components of the underlying graph. In other words, the second eigenvalue is non-zero if and only if the graph is connected. This somewhat surprising connection between the connectivity properties of the graph and the second eigenvalue of the Laplacian can be made even more precise in terms of so-called Cheeger inequalities which also establish a close connection to differential geometry~\cite{CheegerJeff,Fiedler1973,DodziukCheeger,ALON198573,DodziukKendall,MOHAR1989274,MOHAR1988119,Chung2007FourPF,KellerBauerCheeger}. In the same vein, it is well-known in spectral graph theory that the largest eigenvalue of the normalized Laplacian equals two if and only if the graph is bipartite. 
	
The present note mostly falls into the category of spectral graph theory. As we will see, from a spectral point of view, the central object of our discussion is what we call the effective surface area of a graph which is a hybrid version of the (combinatorial) surface area and a so-called external potential associated with the underlying discrete Schrödinger operator. Here, the term Schrödinger operator simply refers to the sum of the normalized Laplacian and some diagonal matrix which represents this external potential. The notion of surface area discussed in this note was introduced recently in~\cite{BK} while studying Schrödinger operators on metric graphs. As we will see in the course of the paper, there is a close connection of this notion to other relevant quantitites of graphs such as the inverse degree and the Randić index~\cite{RandicIndex,Li2008ASO,RInverseDegree}. Also, although the surface area turns out to be, on a formal level, identical to the inverse degree, the important difference lies in the interpretation thereof. Based on the notion of a surface area, we shall then define connectivity measures for graphs which naturally leads to the introduction of a special subclass of large graphs -- so-called social graphs. Roughly speaking, a sequence of (volume increasing) graphs is a sequence of social graphs if and only if the measure of connectivity diverges which happens, for example, if the surface area remains bounded along the sequence. In this sense, social graphs are very large and highly connected graphs and since they are so well-connected, they appear to have a relatively small surface area.

The note is structured as follows: In Section~\ref{SectionSetting} we introduce the relevant graph-theoretical terms as well as the discrete normalized Schrödinger operator whose spectral properties are to studied later on. In Section~\ref{SectionSurfaceArea} we define the (effective) surface area of a graph, we introduce two connectivity measures and define sequences of social graphs. In Section~\ref{Section4} we discuss connections to Cheeger's constant and we apply surgery principles to illustrate the geometrical interpretation of the surface area. In Section~\ref{sec:spectral-estimates} we derive some spectral estimates (upper and lower bounds to eigenvalues) in terms of the surface area and other related quantities such as the Randić index. Most importantly, in Section~\ref{AppendixSpielman}, we derive a novel upper bound on the second eigenvalue for planar graphs which explicitly contains the surface area and which, for some planar graphs, improves existing results by Spielman and Teng \cite{SpTe}, and by Plümer \cite{MP21}. 
	\section{The setting}\label{SectionSetting}
	We study discrete (connected) graphs $\Gamma=\Gamma(E,V)$ with finite vertex set $V$ and finite edge set $E \subset \{\{v,w\} \: : \: v,w \in V \}$. Two vertices $v,w \in V$ are called \emph{adjacent} whenever $\{v,w\} \in E$. Moreover, we call an edge $e \in E$ \emph{incident} to a vertex $v \in V$ whenever there exists another vertex $w \in V$ such that $\{v,w\}  \in E$. We also denote by $n:=|V| \in \mathbb{N}$ the number of vertices. For simplicity, we number the vertices in $V$ by the natural numbers from $1$ to $n$.
	
	In this note  we consider the discrete normalized Schrödinger operator
	\begin{equation}\label{SchrödingerOperator}
	H_U=\mathrm{Id}-D^{-1/2}(A-U)D^{-1/2}
	\end{equation} 
	where $A=(a_{ij}) \in \mathbb{R}^{n \times n}$ is the symmetric \emph{adjacency matrix} of the graph $\Gamma$ and $U=\mathrm{diag}(\sigma_1,\dots,\sigma_n) \in \mathbb{R}^{n \times n}$ denotes a diagonal matrix describing an external potential induced by a set of real (non-negative) parameters $\sigma=(\sigma_1,\dots,\sigma_n) \in \mathbb{R}^n_+:= [0,\infty)$. Note that $\mathrm{Id} \in \mathbb{R}^{n \times n}$ stands for the usual identity matrix on $\mathbb{R}^{n}$. The \emph{degree} of some vertex $j \in V$ is defined via 
$$\deg(j):=\sum_{i=1}^{n}a_{ij}\ ,$$
whereas the \emph{unweighted degree} of some vertex $j \in V$ is defined as 
\[
\deg_{\Gamma}(j) := \sum_{i=1: a_{ij} \neq 0}^{n}(1+\delta_{ij})\ .
\]
Here, $\delta_{ij}$ denotes the usual Kronecker delta. Note that, since we assume the graph to be connected, we also have that $\deg(j) > 0$ for all $j \in V$. Intuitively, the unweighted degree counts the number of edges emanating from a vertex. We call a graph unweighted whenever $a_{ij} \in \{0,1\}$ for all $i,j \in V$ and loop-free whenever $a_{ii}=0$ for all $i \in V$. Based on this, $D=\mathrm{diag}(\deg(1),\dots,\deg(n)) \in \mathbb{R}^{n \times n}$ is the degree matrix.
	
	Endowing  the Euclidean space $\mathbb{R}^n$ with the usual standard inner product $\langle \cdot, \cdot \rangle_{\mathbb{R}^n}$, the quadratic form associated with the self-adjoint linear operator $H_U$ is given by 
	\begin{equation}\label{QuadraticForm}
	\begin{split}
	\langle f,H_U f\rangle_{\mathbb{R}^{n}}&= \|f\|^2_{\mathbb{R}^{n}}-\sum_{i,j=1}^{n}\frac{a_{ij}-\delta_{ij}\sigma_{i}}{\sqrt{\deg(i)\deg(j)}}f_if_j \ ,\\
	&=\frac{1}{2}\sum_{i,j=1}^{n}a_{ij}\left(\frac{f_i}{\sqrt{\deg(i)}}-\frac{f_j}{\sqrt{\deg(j)}}\right)^2+\sum_{i=1}^{n}\frac{\sigma_i}{\deg(i)}f_i^2\ .
	\end{split}
	\end{equation}
    for $f \in \mathbb{R}^{n}$. Later we will work with the form~\eqref{QuadraticForm} in order to obtain bounds on the eigenvalues $\lambda_j(U)$, $j=1,\dots,n$, of $H_U$.
	\section{The surface area of a graph, social graphs and related connectivity measures}\label{SectionSurfaceArea}
	A main focus of this note is on the discussion of two notions of surface area in the context of discrete graphs which were introduced recently in~\cite{BK} studying metric graphs. In that paper, the authors generalized results from~\cite{RWY} obtained for planar domains in $\mathbb{R}^2$ to the setting of metric graphs. More explicitly, comparing the eigenvalues of two self-adjoint realisations of Laplacians on a given domain or (metric) graph, one may derive an explicit expression for the Cesàro average of the differences of the corresponding $n$-th eigenvalues. In the case of a bounded domain $\Omega \subset \mathbb{R}^2$ with smooth enough boundary, this limiting value is given by 
	\begin{equation}\label{LVD}
	\frac{2\sigma|\partial \Omega|}{|\Omega|} \ ,
	\end{equation} 
	where $\sigma > 0$ is a parameter related to the imposed (Robin-)boundary conditions along the boundary $\partial \Omega$ of $\Omega$. For a metric graph of given graph length $\mathcal{L} > 0$ (intuitively, $\mathcal{L}$ is simply the sum of all edge lengths of the metric graph), the analogous limiting value is given by
	\begin{equation}\label{LVG}
	\frac{2}{\mathcal{L}} \sum_{j=1}^{n}\frac{\sigma_j}{\deg(j)} \ ,
	\end{equation}
	where $\sigma_j > 0$, $j \in V$, is a number that determines the matching conditions in the $j$-th vertex (note that, for metric graphs, such matching conditions are necessary in order to obtain a self-adjoint operator. Loosely speaking, one should think of the matching conditions as something originating from external potentials localized on the vertices). Consequently, comparing \eqref{LVD} with \eqref{LVG}, one arrives at the following definition.
	\begin{definition}[Surface area]\label{DefinitionSurfaceArea} Let $\Gamma$ denote a finite discrete graph. Then 
		\begin{equation}\label{Circ}
	\mathcal{S}(\Gamma):=\sum_{j=1}^{n}\frac{1}{\deg(j)}
	\end{equation}
	denotes the surface area of $\Gamma$. For a finite discrete graph $\Gamma$ with external potential $U$, the effective surface area is given by 
	\begin{equation}\label{EffCirc}
	\mathcal{S}_{U}(\Gamma):=\sum_{j=1}^{n}\frac{\sigma_j}{\deg(j)} \ .
	\end{equation}
	\end{definition}
	\begin{remark} Note that the quantities $\mathcal{S}(\Gamma)$ and $\mathcal{S}_U(\Gamma)$ were initially called (effective) circumference in \cite[Remark~6]{BK}. However, since there already exists a different, established and not directly related definition of circumference in classical graph theory, meaning the length of the longest cycle in the underlying graph, we prefer to rename the quantities accordingly. We shall remark that also the term surface area (but with a different meaning) has appeared before in the literature~\cite{IMANI2009560}, dating back to a PhD thesis of B.~Broeg (Oregon State University, 1995), but it does not seem to be canonical.
	\end{remark}
	
It is interesting to observe that the surface area $\mathcal{S}(\Gamma)$ of a graph $\Gamma$ is -- on a formal level -- equivalent to a well-known quantity in graph theory -- the so-called \textit{inverse degree} (cf., for instance, \cite{EPS,DSB,MUKWEMBI2010940}). Therefore, one contribution of the present paper is to provide a new interpretation thereof. Classical results about the inverse degree include an upper bound on the \emph{diameter} $\mathrm{diam}(\Gamma)$ of an unweighted loop-free graph (i.e, the largest combinatorial distance between vertices in the graph) in terms of the inverse degree and the number of vertices. This estimate has been derived in~\cite{EPS} and later improved by Dankelmann, Swart and van den Berg in \cite[Theorem~1]{DSB}: it reads
\[
\mathrm{diam}(\Gamma) \leq (3\mathcal{S}(\Gamma) + 2 + o(1))\frac{\log n}{\log \log n} \ .
\]
It is interesting to remark that this bound grows only very slowly in the number of vertices. At this point we shall also mention that connections of the inverse degree to other quantities of a graph were suggested by the computer program GRAFFITI, leading to the so-called GRAFFITI conjectures~\cite{FAJTLOWICZWallerI,FAJTLOWICZWallerII,FAJTLOWICZGraffiti,ZHANG2004369}.

	In a next step, we establish a theorem which relates the effective surface area of a graph $\Gamma$ to the eigenvalues of $H_U$; this result should be compared with \cite[Theorem~1]{BK} obtained for Schrödinger operators on metric graphs. Quite surprisingly, in the discrete setting considered here, the corresponding statement is a relatively immediate consequence of elementary matrix operations. Furthermore, it is an \textit{exact} relation and not only a statement about limiting values. Nevertheless, as will become clear later, it is a remarkable statement that allows one to attribute a meaningful interpretation to the effective surface area $\mathcal{S}_U(\Gamma)$.
	\begin{theorem}[Trace formula]\label{TraceTheorem} Let $\Gamma$ be a finite discrete graph and $H_U$ a corresponding Schrödinger operator. Then
		\begin{equation*}\begin{split}
		\sum_{j=1}^{n}\lambda_j(U)=n-\sum_{j=1}^{n}\frac{a_{jj}}{\deg(j)}+\mathcal{S}_U(\Gamma)\ ,
		\end{split}
		\end{equation*}
		and 
		\begin{equation*}
		\sum_{j=1}^{n}\left(\lambda_j(U)-\lambda_j(0)\right)=\mathcal{S}_U(\Gamma)\ . 
		\end{equation*}
	\end{theorem}
	\begin{proof} The second identity follows immediately from the first by setting $U=0$ and by observing that $\sum_{j=1}^n \lambda_j(0) = n - \sum_{j=1}^{n}\frac{a_{jj}}{\deg(j)}$. 
		
		To prove the first identity, we start -- noting that the trace is cyclic -- with 
		\begin{equation*}\begin{split}
		\sum_{j=1}^{n}\lambda_j(U)&=\Tr(H_{U})=n-\Tr(D^{-1}(A-U))\ .
		\end{split}
		\end{equation*}
		Furthermore, one observes that
		\begin{equation*}\begin{split}
		\Tr(D^{-1}(A-U))&=\sum_{j=1}^{n}(D^{-1}(A-U))_{jj}=\sum_{i=1}^{n}\sum_{j=1}^{n}(D^{-1})_{ij}(A-U)_{ji} \\
		&=\sum_{j=1}^{n}\frac{a_{jj}-\sigma_{j}}{\deg(j)} = \sum_{j=1}^{n}\frac{a_{jj}}{\deg(j)} - \mathcal{S}_U(\Gamma)\ ,
		\end{split}
		\end{equation*}
		which gives the statement.
	\end{proof}
	\begin{remark}\label{InterEffCirc} Theorem~\ref{TraceTheorem} has an important interpretation which is most visible in the context of large graphs. Imagine a graph $\Gamma$ with a large number of vertices and whose vertex set can be decomposed in a large set of vertices $V_1$ with a large degree and a small set of vertices $V_2$ with a relatively low degree. In such a case, very loosely speaking and assuming the external potential is uniformly bounded,
		\begin{equation*}
		\mathcal{S}_U(\Gamma)\approx \sum_{v \in V_2}\frac{\sigma_j}{\deg(j)}\ .
		\end{equation*}
		Consequently, one concludes that the overall impact of the external potential $U$ on all! eigenvalues of a large graph can be approximated knowing the potential on only a relatively small number of vertices!
	\end{remark}
	From Definition~\ref{DefinitionSurfaceArea} it is not readily clear in what sense $\mathcal{S}(\Gamma)$ can be understood as an area measure. More explicitly, considering a sequence of graphs $(\Gamma_k)_{k \in \mathbb{N}}$ with a growing number of vertices, one would expect the surface area to grow more slowly in the number of vertices than the volume of the graph (which, by definition, is equal to the number of vertices). As will be shown in the following statement, this is indeed not the case for sequences of graphs often considered in the literature (such as planar graphs, graphs with bounded degree, expander graphs).
	\begin{theorem}[Surface area of graphs with bounded average degree]\label{TheoremCircumPlanar} Let $(\Gamma_k)_{k \in \mathbb{N}}$ be a sequence of finite graphs on $(n_k)_{k \in \mathbb{N}}$ vertices and assume that there exists a constant $C > 0$ such that 
		\begin{equation*}
		\frac{\sum_{j \in V_k}\deg(j) }{n_k} \leq C\ , \quad \forall k \in \mathbb{N}\ .
		\end{equation*}
		 Then there exists a constant $c > 0$ such that 
		 \begin{equation*}
		 c \cdot n_k\leq \mathcal{S}(\Gamma_k)
		 \end{equation*}
		 holds for all $k\in \mathbb{N}$.
	\end{theorem}
	\begin{proof} Using the Cauchy-Schwarz inequality, one observes
		\begin{equation*}
		n_k^2=\bigg(\sum_{j \in V_k}1\bigg)^2 \leq \mathcal{S}(\Gamma_k) \cdot \sum_{j \in V_k}\deg(j) \ ,
		\end{equation*}
		which implies 
		\begin{equation*}\begin{split}
		\frac{n_k}{C}\leq \frac{n_k^2}{\sum_{j \in V_{k}}\deg(j) }\leq \mathcal{S}(\Gamma_k)
		\end{split}
		\end{equation*}
		and hence the statement.
	\end{proof}
	%
	%
    Clearly, graphs contained in the setup of Theorem \ref{TheoremCircumPlanar} have degrees that cannot grow too fast. For instance, all unweighted planar graphs have an average degree less than $6$.
	On a more abstract level, Theorem~\ref{TheoremCircumPlanar} shows that for sequences of graphs with bounded average degree, the surface area is of the same size as the volume of the graphs. From this one may conclude that there is no clear distinction between the ``surface'' and the ``interior'' for such graphs. This somewhat negative result then leads in a natural way to the following definition.
	\begin{definition}[Social graphs]\label{SocialGraphs} A sequence $(\Gamma_k)_{k \in \mathbb{N}}$ of finite graphs with associated monotonically increasing sequence of number of vertices $(n_k)_{k \in \mathbb{N}}$ is called a sequence of social (or lump) graphs if 
		\begin{equation*}
		\frac{\mathcal{S}(\Gamma_k)}{n_k} \rightarrow 0 \quad \text{ as} \:\: k \rightarrow \infty\ .
		\end{equation*}
	\end{definition} 
\begin{example}[Complete graphs]\label{ex:complete-graphs} We set $\Gamma_n=K_n$ where $K_n$ is the $n$-complete graph over $n \in \mathbb{N}$ vertices. In this case one has
	\begin{equation*}
	\mathcal{S}(K_n)=\sum_{j=1}^{n}\frac{1}{\deg(j)}=\frac{n}{n-1}\ ,
	\end{equation*} 
	which means that $(K_n)_{n \in \mathbb{N}}$ forms a sequence of social graphs. From this one also concludes that a sequence of social graphs is a sequence where the degrees of the vertices grow fast enough which then implies that the vertices of the graph are highly connected. 
\end{example}
\begin{minipage}[b]{0.33\textwidth}
\begin{center}
\begin{tikzpicture}[scale=0.3]
  \graph { subgraph K_n [n=7,clockwise,radius=1.5cm,empty nodes, nodes={draw, circle, fill=gray, minimum size=.05cm,inner sep=0pt}] };
\end{tikzpicture}
     \captionsetup{font=footnotesize}
     \captionof*{figure}{$K_7$}
\end{center}
\end{minipage}
\begin{minipage}[b]{0.33\textwidth}
\begin{center}
\begin{tikzpicture}[scale=0.3]
  \graph { subgraph K_n [n=9,clockwise,radius=1.5cm,empty nodes, nodes={draw, circle, fill=gray, minimum size=.05cm,inner sep=0pt}] };
\end{tikzpicture}
     \captionsetup{font=footnotesize}
     \captionof*{figure}{$K_9$}
\end{center}
\end{minipage}
\begin{minipage}[b]{0.34\textwidth}
\begin{center}
\begin{tikzpicture}[scale=0.3]
  \graph { subgraph K_n [n=16,clockwise,radius=1.5cm,empty nodes, nodes={draw, circle, fill=gray, minimum size=.05cm,inner sep=0pt}] };
\end{tikzpicture}
     \captionsetup{font=footnotesize}
     \captionof*{figure}{$K_{16}$}
\end{center}
\end{minipage}
	Since the surface area of a graph is small whenever the degrees of the vertices are large, it seems plausible to introduce \textit{measures of connectivity} based on the surface area.
	\begin{definition}[Connectivity measure]\label{MeasureConnectivity} For a finite graph $\Gamma$ on $n \in \mathbb{N}$ vertices we introduce
		\begin{equation*}
		\mathcal{C}(\Gamma):=\frac{n}{\mathcal{S}(\Gamma)}
		\end{equation*}
		as a measure of connectivity of the graph. Likewise, one defines $\mathcal{C}_U(\Gamma)$ using $\mathcal{S}_U(\Gamma)$.
	\end{definition}
	Comparing Definition~\ref{MeasureConnectivity} with Definition~\ref{SocialGraphs} one concludes that $(\Gamma_k)_{k \in \mathbb{N}}$ forms a sequence of social graphs if and only if $\lim_{k \rightarrow \infty}\mathcal{C}(\Gamma_k)=\infty$. This supports the intuition of a social graphs as a highly connected graphs.

\section{On the connection of the surface area to other geometrical quantities of a graph}\label{Section4}
\subsection{Relation to the Cheeger constant}
In a first step we want to discuss the connection between the well-known Cheeger constant $h(\Gamma)$ of a graph $\Gamma$ and its surface area $\mathcal{S}(\Gamma)$. For this, we recall that 
\begin{equation}\label{CheegerConstant}
h(\Gamma):=\min_{X \subset V}\frac{E(X,V\setminus X)}{\min \{\mathrm{vol}(X),\mathrm{vol}(V\setminus X) \}}\ ,
\end{equation}
where $E(X,V\setminus X):=\sum_{i \in X, j \in V\setminus X}a_{ij}$ denotes the weighted number of edges connecting vertices in $X$ and its complement and $\mathrm{vol}(X):=\sum_{j \in X}\deg(j)$ \cite{Chung,KELLER201680}. Intuitively, the Cheeger constant of a connected graph is small whenever it is possible to cut the graph into two (equally large) parts without cutting through too many edges.

\begin{prop}\label{CheegerLarge} There exists a sequence of social graphs $(\Gamma_k)_{k \in \mathbb{N}}$ for which $h(\Gamma_k) \rightarrow 0$ as $k \rightarrow \infty$ and there exists another sequence of social graphs $(\hat{\Gamma}_k)_{k \in \mathbb{N}}$ for which $h(\hat{\Gamma}_k) \geq \varepsilon$ for some $\varepsilon > 0$ and all $k \in \mathbb{N}$.
\end{prop}
\begin{proof} To prove the second statement, we set $\hat{\Gamma}_n:=K_n$ where $K_n$ is the unweighted $n$-complete graph on $n$ vertices as in Example \ref{ex:complete-graphs} and recall that $h(K_n)$ is well-known to be bounded away from zero: for instance, one may observe this using Cheeger's inequality (also compare with the proof of Theorem~\ref{TheoremUpperBound}) and the fact that the second eigenvalue of the normalized Laplacian associated to $K_n$ is given by $\frac{n}{n-1} \geq 1$ (with multiplicity $n-1$), cf.\ \cite[Example~1.1 and Lemma~2.1]{Chung}. 

	With respect to the first statement, we choose a sequence $(\Gamma_n)_{n \in \mathbb{N}}$ with $\Gamma_n$ to be the (unweighted) graph obtained by connecting two disjoint copies of $K_n$ by one fixed edge. 
 \begin{center}
 \vspace{-0.2cm}
 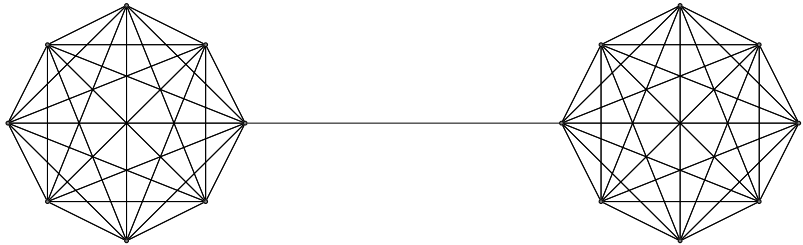
\begin{figure}[h]

\begin{tikzpicture}[scale=0.45]
      \tikzset{enclosed/.style={draw, circle, inner sep=0pt, minimum size=.05cm, fill=gray}, every loop/.style={}}

      \node[enclosed] (A) at (-4,4) {};
      \node[enclosed] (B) at (0,4) {};
      \node[enclosed] (C) at (0,0) {};
      \node[enclosed] (D) at (-4,0) {};
       \node[enclosed] (E) at (-2,5) {};
      \node[enclosed] (F) at (1,2) {};
      \node[enclosed] (G) at (-2,-1) {};
      \node[enclosed] (H) at (-5,2) {};

      \node[enclosed] (A') at (10,4) {};
      \node[enclosed] (B') at (14,4) {};
      \node[enclosed] (C') at (14,0) {};
      \node[enclosed] (D') at (10,0) {};
       \node[enclosed] (E') at (12,5) {};
      \node[enclosed] (F') at (15,2) {};
      \node[enclosed] (G') at (12,-1) {};
      \node[enclosed] (H') at (9,2) {};
      
      \draw[line width=0.1pt] (F) edge node[above] {} (H') node[midway, above] (edge1) {};
      \draw[line width=0.1pt] (A) edge node[above] {} (B) node[midway, above] (edge1) {};
      \draw[line width=0.1pt] (A) edge node[above] {} (C) node[midway, above] (edge2) {};
      \draw[line width=0.1pt] (A) edge node[above] {} (D) node[midway, above] (edge3) {};
      \draw[line width=0.1pt] (A) edge node[above] {} (E) node[midway, above] (edge4) {};
      \draw[line width=0.1pt] (A) edge node[above] {} (F) node[midway, above] (edge1) {};
      \draw[line width=0.1pt] (A) edge node[above] {} (G) node[midway, above] (edge2) {};
      \draw[line width=0.1pt] (A) edge node[above] {} (H) node[midway, above] (edge3) {};
      \draw[line width=0.1pt] (B) edge node[above] {} (A) node[midway, above] (edge1) {};
      \draw[line width=0.1pt] (B) edge node[above] {} (C) node[midway, above] (edge2) {};
      \draw[line width=0.1pt] (B) edge node[above] {} (D) node[midway, above] (edge3) {};
      \draw[line width=0.1pt] (B) edge node[above] {} (E) node[midway, above] (edge4) {};
      \draw[line width=0.1pt] (B) edge node[above] {} (F) node[midway, above] (edge1) {};
      \draw[line width=0.1pt] (B) edge node[above] {} (G) node[midway, above] (edge2) {};
      \draw[line width=0.1pt] (B) edge node[above] {} (H) node[midway, above] (edge3) {};
      \draw[line width=0.1pt] (C) edge node[above] {} (A) node[midway, above] (edge1) {};
      \draw[line width=0.1pt] (C) edge node[above] {} (B) node[midway, above] (edge2) {};
      \draw[line width=0.1pt] (C) edge node[above] {} (D) node[midway, above] (edge3) {};
      \draw[line width=0.1pt] (C) edge node[above] {} (E) node[midway, above] (edge4) {};
      \draw[line width=0.1pt] (C) edge node[above] {} (F) node[midway, above] (edge1) {};
      \draw[line width=0.1pt] (C) edge node[above] {} (G) node[midway, above] (edge2) {};
      \draw[line width=0.1pt] (C) edge node[above] {} (H) node[midway, above] (edge3) {};
      \draw[line width=0.1pt] (D) edge node[above] {} (A) node[midway, above] (edge1) {};
      \draw[line width=0.1pt] (D) edge node[above] {} (B) node[midway, above] (edge2) {};
      \draw[line width=0.1pt] (D) edge node[above] {} (C) node[midway, above] (edge3) {};
      \draw[line width=0.1pt] (D) edge node[above] {} (E) node[midway, above] (edge4) {};
      \draw[line width=0.1pt] (D) edge node[above] {} (F) node[midway, above] (edge1) {};
      \draw[line width=0.1pt] (D) edge node[above] {} (G) node[midway, above] (edge2) {};
      \draw[line width=0.1pt] (D) edge node[above] {} (H) node[midway, above] (edge3) {};
      \draw[line width=0.1pt] (E) edge node[above] {} (A) node[midway, above] (edge1) {};
      \draw[line width=0.1pt] (E) edge node[above] {} (B) node[midway, above] (edge2) {};
      \draw[line width=0.1pt] (E) edge node[above] {} (C) node[midway, above] (edge3) {};
      \draw[line width=0.1pt] (E) edge node[above] {} (D) node[midway, above] (edge4) {};
      \draw[line width=0.1pt] (E) edge node[above] {} (F) node[midway, above] (edge1) {};
      \draw[line width=0.1pt] (E) edge node[above] {} (G) node[midway, above] (edge2) {};
      \draw[line width=0.1pt] (E) edge node[above] {} (H) node[midway, above] (edge3) {};
      \draw[line width=0.1pt] (F) edge node[above] {} (A) node[midway, above] (edge1) {};
      \draw[line width=0.1pt] (F) edge node[above] {} (B) node[midway, above] (edge2) {};
      \draw[line width=0.1pt] (F) edge node[above] {} (C) node[midway, above] (edge3) {};
      \draw[line width=0.1pt] (F) edge node[above] {} (D) node[midway, above] (edge4) {};
      \draw[line width=0.1pt] (F) edge node[above] {} (E) node[midway, above] (edge1) {};
      \draw[line width=0.1pt] (F) edge node[above] {} (G) node[midway, above] (edge2) {};
      \draw[line width=0.1pt] (F) edge node[above] {} (H) node[midway, above] (edge3) {};
      \draw[line width=0.1pt] (G) edge node[above] {} (A) node[midway, above] (edge1) {};
      \draw[line width=0.1pt] (G) edge node[above] {} (B) node[midway, above] (edge2) {};
      \draw[line width=0.1pt] (G) edge node[above] {} (C) node[midway, above] (edge3) {};
      \draw[line width=0.1pt] (G) edge node[above] {} (D) node[midway, above] (edge4) {};
      \draw[line width=0.1pt] (G) edge node[above] {} (E) node[midway, above] (edge1) {};
      \draw[line width=0.1pt] (G) edge node[above] {} (F) node[midway, above] (edge2) {};
      \draw[line width=0.1pt] (G) edge node[above] {} (H) node[midway, above] (edge3) {};
      \draw[line width=0.1pt] (H) edge node[above] {} (A) node[midway, above] (edge1) {};
      \draw[line width=0.1pt] (H) edge node[above] {} (B) node[midway, above] (edge2) {};
      \draw[line width=0.1pt] (H) edge node[above] {} (C) node[midway, above] (edge3) {};
      \draw[line width=0.1pt] (H) edge node[above] {} (D) node[midway, above] (edge4) {};
      \draw[line width=0.1pt] (H) edge node[above] {} (E) node[midway, above] (edge1) {};
      \draw[line width=0.1pt] (H) edge node[above] {} (F) node[midway, above] (edge2) {};
      \draw[line width=0.1pt] (H) edge node[above] {} (G) node[midway, above] (edge3) {};

      \draw[line width=0.1pt] (A') edge node[above] {} (B') node[midway, above] (edge1) {};
      \draw[line width=0.1pt] (A') edge node[above] {} (C') node[midway, above] (edge2) {};
      \draw[line width=0.1pt] (A') edge node[above] {} (D') node[midway, above] (edge3) {};
      \draw[line width=0.1pt] (A') edge node[above] {} (E') node[midway, above] (edge4) {};
      \draw[line width=0.1pt] (A') edge node[above] {} (F') node[midway, above] (edge1) {};
      \draw[line width=0.1pt] (A') edge node[above] {} (G') node[midway, above] (edge2) {};
      \draw[line width=0.1pt] (A') edge node[above] {} (H') node[midway, above] (edge3) {};
      \draw[line width=0.1pt] (B') edge node[above] {} (A') node[midway, above] (edge1) {};
      \draw[line width=0.1pt] (B') edge node[above] {} (C') node[midway, above] (edge2) {};
      \draw[line width=0.1pt] (B') edge node[above] {} (D') node[midway, above] (edge3) {};
      \draw[line width=0.1pt] (B') edge node[above] {} (E') node[midway, above] (edge4) {};
      \draw[line width=0.1pt] (B') edge node[above] {} (F') node[midway, above] (edge1) {};
      \draw[line width=0.1pt] (B') edge node[above] {} (G') node[midway, above] (edge2) {};
      \draw[line width=0.1pt] (B') edge node[above] {} (H') node[midway, above] (edge3) {};
      \draw[line width=0.1pt] (C') edge node[above] {} (A') node[midway, above] (edge1) {};
      \draw[line width=0.1pt] (C') edge node[above] {} (B') node[midway, above] (edge2) {};
      \draw[line width=0.1pt] (C') edge node[above] {} (D') node[midway, above] (edge3) {};
      \draw[line width=0.1pt] (C') edge node[above] {} (E') node[midway, above] (edge4) {};
      \draw[line width=0.1pt] (C') edge node[above] {} (F') node[midway, above] (edge1) {};
      \draw[line width=0.1pt] (C') edge node[above] {} (G') node[midway, above] (edge2) {};
      \draw[line width=0.1pt] (C') edge node[above] {} (H') node[midway, above] (edge3) {};
      \draw[line width=0.1pt] (D') edge node[above] {} (A') node[midway, above] (edge1) {};
      \draw[line width=0.1pt] (D') edge node[above] {} (B') node[midway, above] (edge2) {};
      \draw[line width=0.1pt] (D') edge node[above] {} (C') node[midway, above] (edge3) {};
      \draw[line width=0.1pt] (D') edge node[above] {} (E') node[midway, above] (edge4) {};
      \draw[line width=0.1pt] (D') edge node[above] {} (F') node[midway, above] (edge1) {};
      \draw[line width=0.1pt] (D') edge node[above] {} (G') node[midway, above] (edge2) {};
      \draw[line width=0.1pt] (D') edge node[above] {} (H') node[midway, above] (edge3) {};
      \draw[line width=0.1pt] (E') edge node[above] {} (A') node[midway, above] (edge1) {};
      \draw[line width=0.1pt] (E') edge node[above] {} (B') node[midway, above] (edge2) {};
      \draw[line width=0.1pt] (E') edge node[above] {} (C') node[midway, above] (edge3) {};
      \draw[line width=0.1pt] (E') edge node[above] {} (D') node[midway, above] (edge4) {};
      \draw[line width=0.1pt] (E') edge node[above] {} (F') node[midway, above] (edge1) {};
      \draw[line width=0.1pt] (E') edge node[above] {} (G') node[midway, above] (edge2) {};
      \draw[line width=0.1pt] (E') edge node[above] {} (H') node[midway, above] (edge3) {};
      \draw[line width=0.1pt] (F') edge node[above] {} (A') node[midway, above] (edge1) {};
      \draw[line width=0.1pt] (F') edge node[above] {} (B') node[midway, above] (edge2) {};
      \draw[line width=0.1pt] (F') edge node[above] {} (C') node[midway, above] (edge3) {};
      \draw[line width=0.1pt] (F') edge node[above] {} (D') node[midway, above] (edge4) {};
      \draw[line width=0.1pt] (F') edge node[above] {} (E') node[midway, above] (edge1) {};
      \draw[line width=0.1pt] (F') edge node[above] {} (G') node[midway, above] (edge2) {};
      \draw[line width=0.1pt] (F') edge node[above] {} (H') node[midway, above] (edge3) {};
      \draw[line width=0.1pt] (G') edge node[above] {} (A') node[midway, above] (edge1) {};
      \draw[line width=0.1pt] (G') edge node[above] {} (B') node[midway, above] (edge2) {};
      \draw[line width=0.1pt] (G') edge node[above] {} (C') node[midway, above] (edge3) {};
      \draw[line width=0.1pt] (G') edge node[above] {} (D') node[midway, above] (edge4) {};
      \draw[line width=0.1pt] (G') edge node[above] {} (E') node[midway, above] (edge1) {};
      \draw[line width=0.1pt] (G') edge node[above] {} (F') node[midway, above] (edge2) {};
      \draw[line width=0.1pt] (G') edge node[above] {} (H') node[midway, above] (edge3) {};
      \draw[line width=0.1pt] (H') edge node[above] {} (A') node[midway, above] (edge1) {};
      \draw[line width=0.1pt] (H') edge node[above] {} (B') node[midway, above] (edge2) {};
      \draw[line width=0.1pt] (H') edge node[above] {} (C') node[midway, above] (edge3) {};
      \draw[line width=0.1pt] (H') edge node[above] {} (D') node[midway, above] (edge4) {};
      \draw[line width=0.1pt] (H') edge node[above] {} (E') node[midway, above] (edge1) {};
      \draw[line width=0.1pt] (H') edge node[above] {} (F') node[midway, above] (edge2) {};
      \draw[line width=0.1pt] (H') edge node[above] {} (G') node[midway, above] (edge3) {};
     \end{tikzpicture}
     \caption{Two copies of $K_8$ connected by a single edge.}
     \end{figure}
\end{center}
\vspace{-0.7cm}
 Putting $X_n$ as the set of vertices of the first copy of $K_n$ such that $V_n \setminus X_n$ consists of all vertices of the second copy of $K_n$, we observe that for every $n \in \mathbb{N}$ one has that $E(X_n,V_n \setminus X_n) = 1$ as well as
 \[
 \mathrm{vol}(X_n) = \mathrm{vol}(V_n \setminus X_n) = (n-1)^2 + n \ .
 \]
 Thus, according to the definition of the Cheeger constant, it follows that
	\begin{equation*}
	h(\Gamma_n) \leq \frac{1}{(n-1)^2+n} \rightarrow 0 \quad \text{as} \:\: n \rightarrow \infty \ ,
	\end{equation*}
	from which the statement follows.
\end{proof}
\begin{remark} In view of Proposition~\ref{CheegerLarge} it is instructive to mention that, for a sequence of so-called expander graphs $(\Gamma_n)_{n \in \mathbb{N}}$, by definition one always has $h(\Gamma_n) \geq \varepsilon$ for some $\varepsilon > 0$ but also $\mathcal{S}(\Gamma_n) \rightarrow \infty$ since the vertex degrees are, also by definition, uniformly bounded. 
\end{remark}

\subsection{Surgery of graphs and the surface area}\label{AppendixGeometric}
In this section we investigate the impact of basic surgery methods such as gluing and cutting at vertices resp.\:edges on the surface area $\mathcal{S}(\Gamma)$. Note that such surgery methods have been thoroughly studied in the context of metric graphs, cf.\ for instance \cite{BKKMSurgery}. As we will see, from a geometrical point of view, this support the interpretation of the inverse degree as a surface area.
\begin{lemma}[Gluing of two graphs together at a vertex]\label{LemmaGluing}
Let $\Gamma=\Gamma(V_1,E_1)$ and $\Gamma'=\Gamma'(V_2,E_2)$ be two graphs with $i_1 \in V_1$ as well as $i_2 \in V_2$. Furthermore, let $\Gamma \cup \Gamma'$ denote the graph which arises by gluing the graphs $\Gamma$ and $\Gamma'$ together at the vertices $i_1$ and $i_2$ (see Figure \ref{fig:gluingvertices}). Then,
\[
\mathcal{S}(\Gamma \cup \Gamma') \leq \mathcal{S}(\Gamma) + \mathcal{S}(\Gamma') \ .
\]
\end{lemma}
\begin{figure}[h]
\begin{tikzpicture}[scale=0.45]
      \tikzset{enclosed/.style={draw, circle, inner sep=0pt, minimum size=.1cm, fill=gray}}

      \node[enclosed, label={}] (A) at (0,5.5) {};
      \node[enclosed, label={}] (C) at (2,4) {};
      \node[enclosed, label={}] (B) at (0,2.5) {};
      \node[enclosed, label={$i_1$}] (D) at (4,4) {};
      \node[enclosed, label={below: $\Gamma$},white] (W1) at (2,2.5) {};

      \node[enclosed, label={$i_2$}] (A') at (6,4) {};
      \node[enclosed, label={}] (B') at (8,4) {};
      \node[enclosed, label={}] (C') at (10,5.5) {};
      \node[enclosed, label={}] (D') at (10,4) {};
      \node[enclosed, label={}] (E') at (10,2.5) {};
      \node[enclosed, label={below: $\Gamma'$},white] (W1') at (8,2.5) {};

      \node[enclosed, label={},white] (W2) at (11,4) {};
      \node[enclosed, label={},white] (W3) at (13,4) {};

      \node[enclosed, label={}] (A'') at (14,5.5) {};
      \node[enclosed, label={}] (C'') at (16,4) {};
      \node[enclosed, label={}] (B'') at (14,2.5) {};
      \node[enclosed, label={$i_1 \simeq i_2$}] (D'') at (18,4) {};      
      \node[enclosed, label={}] (B''') at (20,4) {};
      \node[enclosed, label={}] (C''') at (22,5.5) {};
      \node[enclosed, label={}] (D''') at (22,4) {};
      \node[enclosed, label={}] (E''') at (22,2.5) {};
    \node[enclosed, label={below: $\Gamma \cup \Gamma'$},white] (W1'') at (18,2.5) {};

      \draw (A) edge node[below] {} (C) node[midway, above] (edge1) {};
      \draw (B) edge node[below] {} (C) node[midway, above] (edge2) {};
      \draw (C) edge node[below] {} (D) node[midway, above] (edge3) {};

      \draw (A') edge node[below] {} (B') node[midway, above] (edge1') {};
      \draw (B') edge node[below] {} (C') node[midway, above] (edge2') {};
      \draw (B') edge node[below] {} (D') node[midway, above] (edge3') {};
      \draw (B') edge node[below] {} (E') node[midway, above] (edge4') {};

      \draw[ultra thick, ->] (W2) edge node[below] {} (W3) node[midway, above] (arrow) {};

      \draw (A'') edge node[below] {} (C'') node[midway, above] (edge1'') {};
      \draw (B'') edge node[below] {} (C'') node[midway, above] (edge2'') {};
      \draw (C'') edge node[below] {} (D'') node[midway, above] (edge3'') {};
      \draw (D'') edge node[below] {} (B''') node[midway, above] (edge4'') {};
      \draw (B''') edge node[below] {} (C''') node[midway, above] (edge5'') {};
      \draw (B''') edge node[below] {} (D''') node[midway, above] (edge6'') {};
      \draw (B''') edge node[below] {} (E''') node[midway, above] (edge7'') {};

     \end{tikzpicture}
     \vspace{-0.3cm}
     \caption{Gluing the graphs $\Gamma$ and $\Gamma'$ together at $i_1$ and $i_2$}\label{fig:gluingvertices}
\end{figure}
\begin{proof}
Let us identify $V_1=\{1,\dots,|V_1|\}$ and $V_2:=\{|V_1|,\dots,|V_1|+|V_2|-1\}$ such that $i_1=|V_1|=i_2$. The vertex set of $V_{\Gamma \cup \Gamma'}$ can then be written as $\{1,\dots,|V_1|+|V_2|-1\}$. For all vertices $i \in V_{\Gamma \cup \Gamma'}$ with $i\neq |V_1|$, the degree is not affected by gluing. Furthermore, for $i=|V_1|$ and using an obvious notation, one has
\[
\deg(i \in V_{\Gamma \cup \Gamma'}) = \deg(i \in V_{\Gamma}) + \deg(i \in V_{\Gamma'})
\]
and hence we obtain 
\begin{align*}
\frac{1}{\deg(i \in V_{\Gamma \cup \Gamma'})} &= \frac{1}{\deg(i \in V_{\Gamma}) + \deg(i \in V_{\Gamma'})} \\&\leq \frac{1}{\deg(i \in V_{\Gamma}) + \deg(i \in V_{\Gamma'})} + \frac{1}{\deg(i \in V_{\Gamma'})} \\&\leq \frac{1}{\deg(i \in V_{\Gamma})} + \frac{1}{\deg(i \in V_{\Gamma'})}\ .
\end{align*}
Since all other degrees remain unchanged, this gives $\mathcal{S}(\Gamma \cup \Gamma') \leq \mathcal{S}(\Gamma') + \mathcal{S}(\Gamma)$.
\end{proof}
In agreement with standard geometrical intuition that stems from domains in $\mathbb{R}^2$, \linebreak Lemma~\ref{LemmaGluing} shows that the gluing of two graphs leads to a decreasing surface area. In the same vein, one could now also ask whether separating a graph into two smaller graphs also follows the expected lines; this is indeed the case as demonstrated by the following statement.
\begin{lemma}[Cutting a graph]
Let $\Gamma = \Gamma(V,E)$ be some connected graph with $e \in E$ and such that the graph $\Gamma'=\Gamma'(V,E')$ with $E' := E \setminus \{ e \}$ is disconnected and the union of two connected graphs $\Gamma_1 = (E_1,V_1)$ and $\Gamma_2 = (E_2,V_2)$ (cf.~Figure \ref{fig:cuttingedges}). Then
\[
\mathcal{S}(\Gamma) \leq \mathcal{S}(\Gamma_1) + \mathcal{S}(\Gamma_2) \ .
\]
\end{lemma}
\begin{figure}[h]
\begin{tikzpicture}[scale=0.50]
      \tikzset{enclosed/.style={draw, circle, inner sep=0pt, minimum size=.1cm, fill=gray}}

      \node[enclosed, label={}] (A) at (0,5.5) {};
      \node[enclosed, label={}] (C) at (2,4) {};
      \node[enclosed, label={}] (B) at (0,2.5) {};
      \node[enclosed, label={}] (D) at (5,4) {};
      \node[enclosed, label={below: $\Gamma$},white] (W1) at (3.5,2.5) {};
      \node[enclosed, label={}] (C') at (7,5.5) {};
      \node[enclosed, label={}] (D') at (7,4) {};
      \node[enclosed, label={}] (E') at (7,2.5) {};

      \node[enclosed, label={},white] (W2) at (8,4) {};
      \node[enclosed, label={},white] (W3) at (10,4) {};

      \node[enclosed, label={}] (A'') at (11,5.5) {};
      \node[enclosed, label={}] (C'') at (13,4) {};
      \node[enclosed, label={}] (B'') at (11,2.5) {};
      
      \node[enclosed, label={}] (B''') at (16,4) {};
      \node[enclosed, label={}] (C''') at (18,5.5) {};
      \node[enclosed, label={}] (D''') at (18,4) {};
      \node[enclosed, label={}] (E''') at (18,2.5) {};
        \node[enclosed, label={below: $\Gamma_1$},white] (W1'') at (12,2.5) {};
        \node[enclosed, label={below: $\Gamma_2$},white] (W2'') at (17,2.5) {};

      \draw (A) edge node[below] {} (C) node[midway, above] (edge1) {};
      \draw (B) edge node[below] {} (C) node[midway, above] (edge2) {};
      \draw[dashed] (C) edge node[below] {$e$} (D) node[midway, above] (edge3) {};

      \draw (D) edge node[below] {} (C') node[midway, above] (edge2') {};
      \draw (D) edge node[below] {} (D') node[midway, above] (edge3') {};
      \draw (D) edge node[below] {} (E') node[midway, above] (edge4') {};

      \draw[ultra thick, ->] (W2) edge node[below] {} (W3) node[midway, above] (arrow) {};

      \draw (A'') edge node[below] {} (C'') node[midway, above] (edge1'') {};
      \draw (B'') edge node[below] {} (C'') node[midway, above] (edge2'') {};
      \draw (B''') edge node[below] {} (C''') node[midway, above] (edge5'') {};
      \draw (B''') edge node[below] {} (D''') node[midway, above] (edge6'') {};
      \draw (B''') edge node[below] {} (E''') node[midway, above] (edge7'') {};
     \end{tikzpicture}
     \vspace{-0.3cm}
     \caption{Cutting the graph at the edge $e$ into disconnected graphs $\Gamma_1$ and $\Gamma_2$.}\label{fig:cuttingedges}
\end{figure}
\begin{proof} This statement is a simple consequence of the fact that, for any graph $\Gamma$, the deletion of an edge reduces the degree of the associated vertices and hence the surface are increases.
\end{proof}
Also, adding a pending edge to some graph increases the surface area as shown in the next statement.
\begin{lemma}[Attaching pending edges]\label{lem:attachingpendingedges}\label{LemmaAddingPendingEdge}
Let $\Gamma = \Gamma(V,E)$ be a un unweighted graph and assume that $\Gamma' = \Gamma'(V',E')$ is unweighted and obtained by attaching an additional edge $e' \in E' \setminus E$ together with a (boundary) vertex $i' \in V' \setminus V$ (of degree one) to some $i \in V$ (as in Figure \ref{fig:pendingedges}); in particular, we have that $V'=V \,\dot{\cup}\, \{i'\}$ and $E' = E \,\dot{\cup}\, \{e'\}$. Then 
$$\mathcal{S}(\Gamma) \leq \mathcal{S}(\Gamma')\ .$$
\end{lemma} 
\begin{figure}[h]
\begin{tikzpicture}[scale=0.4]
      \tikzset{enclosed/.style={draw, circle, inner sep=0pt, minimum size=.1cm, fill=gray}}

      \node[enclosed, label={}] (C) at (4,6) {};
      \node[enclosed, label={$i$}] (A) at (8,4) {};
      \node[enclosed, label={}] (B) at (0,4) {};
      \node[enclosed, label={}] (D) at (4,2) {};
      \node[enclosed, label={$i'$}] (P) at (11,4) {};

      \draw (A) edge node[below] {} (C) node[midway, above] (edge1) {};
      \draw (B) edge node[below] {} (C) node[midway, above] (edge2) {};
      \draw (B) edge node[below] {} (D) node[midway, above] (edge3) {};
      \draw (A) edge node[below] {} (D) node[midway, above] (edge4) {};
      \draw[dashed] (A) edge node[below] {$e'$} (P) node[midway, above] (edge5) {};
     \end{tikzpicture}
     \vspace{-0.3cm}
     \caption{A \emph{diamond graph} $\Gamma$ with a pendant edge $e' = (i,i')$ in $v$.}\label{fig:pendingedges}
\end{figure}
\begin{proof}
Clearly, for every $j \neq i$ we have that $\deg_{\Gamma}(j)=\deg_{\Gamma'}(j)$. Thus, we only have to consider the vertex $i \in V$ itself. As $\deg_{\Gamma}(i), \deg_{\Gamma'}(i) \in \mathbb{N}$ and by construction $\deg_{\Gamma'}(i) = \deg_{\Gamma}(i)+1$, we observe
\[
\frac{\deg_{\Gamma'}(i)}{\deg_{\Gamma}(i)} = \frac{\deg_{\Gamma}(i)+1}{\deg_{\Gamma}(i)} = 1 + \frac{1}{\deg_{\Gamma}(i)} \leq 1+\deg_{\Gamma'}(i)
\]
which implies by dividing $\deg_{\Gamma'}(i)$ on both sides that
\[
\frac{1}{\deg_{\Gamma}(i)} \leq \frac{1}{\deg_{\Gamma'}(i)} + 1 =  \frac{1}{\deg_{\Gamma'}(i)} + \frac{1}{\deg_{\Gamma'}(i')}.
\]
This yields $\mathcal{S}(\Gamma) \leq \mathcal{S}(\Gamma')$.
\end{proof}
Using Lemma~\ref{LemmaAddingPendingEdge} recursively, the following statement follows readily.
\begin{proposition} Let $\Gamma=\Gamma(V,E)$ be a connected, unweighted \emph{graph} and $\Gamma' = \Gamma(V',E')$ an unweighted graph that is constructed from $\Gamma$  by recursively adding pending edges as in Lemma~\ref{LemmaAddingPendingEdge}. Then 
$$\mathcal{S}(\Gamma) \leq \mathcal{S}(\Gamma')\ .$$
\end{proposition}

\section{Spectral estimates}\label{sec:spectral-estimates}
The goal of this section is to derive spectral bounds on certain eigenvalues of $H_U$ involving the surface area of a graph and related quantities. Recall that our main spectral bound is derived in Section~\ref{AppendixSpielman}.
\subsection{Bounds on the second eigenvalue $\lambda_2(U)$}\label{sec:bound-sec-eig}
Based on the Cheeger constant defined in~\eqref{CheegerConstant} and using classical estimates, we immediately arrive at the following statement.

\begin{theorem}[An upper bound on $\lambda_2(U)$]\label{TheoremUpperBound} For $U \geq 0$ one has the estimate
	\begin{equation*}
	\lambda_2(U) \leq \lambda_2(0)+\mathcal{S}_U(\Gamma) \leq 2h(\Gamma)+\mathcal{S}_U(\Gamma) \ .
	\end{equation*}
\end{theorem}
\begin{proof} We simply observe that the min-max principle implies $\lambda_2(U) \leq \lambda_2(0)+\mathcal{S}_U(\Gamma)$ and combine it with the classical estimate $\lambda_2(0) \leq 2h(\Gamma)$~(see, e.g., \cite{Chung}).
\end{proof}

\subsection{Bounds on the largest eigenvalue $\lambda_n(U)$}\label{sec:randic-bound}
In this section we shall derive a lower bound on the largest eigenvalue of $H_{U}$ which is expressed in terms of the surface area $\mathcal{S}(\Gamma)$, the effective surface area $\mathcal{S}_U(\Gamma)$ and versions of the so-called \textit{Randi\'{c} index}~\cite{RandicIndex,EBRandic}. More explicitly, for $\alpha \in \mathbb{R}$ we set
\begin{equation*}
R_{\alpha}(\Gamma):=\frac{1}{2}\sum_{i,j = 1}^{n}a_{ij}(\deg(i)\deg(j))^{\alpha}\ , 
\end{equation*}
and note that this agrees with the standard Randi\'{c} index in the unweighted case (also note that Randi\'{c} originally considered the case $\alpha=-\frac{1}{2}$). For two subsets $V_a$ and $V_b$ of the vertex set $V$, we also introduce the \textit{restricted Randi\'{c} index} via
\begin{equation*}
R_{\alpha}(\Gamma,V_a\times V_b):=\frac{1}{2}\sum_{i,j=1}^{n}a_{ij}\chi_{V_a}(i)\chi_{V_b}(j)(\deg(i)\deg(j))^{\alpha}\ , 
\end{equation*}
where $\chi_{\hat{V}}:V \rightarrow \{0,1\}$ denotes the characteristic function of a vertex subset $\hat{V} \subset V$; we abbreviate $R_\alpha(\Gamma, \hat{V}) := R_\alpha(\Gamma, \hat{V} \times \hat{V})$ for every subset $\hat{V} \subset V$.

Using the well-known min-max principle, one can characterize the $n$-th eigenvalue $\lambda_n(U)$ via maximizing the so-called \emph{Rayleigh quotient} which leads to the relation
\begin{equation}\label{MinMax}
\lambda_n(U)=\max_{f \in \mathbb{R}^{n} \setminus \{ 0 \}}\frac{\langle f,H_{U} f\rangle_{\mathbb{R}^n}}{\|f\|^2_{\mathbb{R}^n}}\ .
\end{equation}
It is well-known that $\lambda_n(0) \leq 2$ (since we consider the normalized Laplacian) and hence, as in Theorem \ref{TheoremUpperBound}, a first upper bound on $\lambda_n(U)$ is given by
\begin{equation*}
\lambda_n(U)\leq 2+\mathcal{S}_{U}(\Gamma)
\end{equation*}
which also implies $\sigma(H_{U}) \subset \left[0,2+\mathcal{S}_{U}(\Gamma)\right]$. We then set, for $\alpha \in \mathbb{R}$, 
\begin{equation*}
\mathcal{S}_U(\Gamma,\alpha):=\sum_{j=1}^{n}\frac{\sigma_j}{\deg(j)^{\alpha}}
\end{equation*}
and obtain the following statement.
\begin{theorem}[A lower bound on $\lambda_n(U)$]\label{LowerBound} For a graph $\Gamma$ we have
	\begin{equation*}
	\max_{V_a,V_b \subset V} \left(1+\frac{1}{\mathcal{S}(\Gamma)}\big(2R_{-1}(\Gamma,V_a\times V_b)-R_{-1}(\Gamma,V_a)-R_{-1}(\Gamma,V_b)\big)+\frac{\mathcal{S}_U(\Gamma,2)}{\mathcal{S}(\Gamma)}\right) \leq \lambda_n(U)\ ,
	\end{equation*}
	where $V_a\cap V_b=\emptyset$ and $V_a \cup V_b=V$. 
\end{theorem}
\begin{proof} Let $V_a$ and $V_b$ be an arbitrary disjoint decomposition of the vertex set $V$. We insert the vector $f \in \mathbb{R}^n$ defined via
	\begin{equation*}
	f_j:=\pm \frac{1}{\sqrt{\deg(j)}}\ , \quad j=1,\dots,n\ ,
	\end{equation*}
	into the Rayleigh quotient in \eqref{MinMax}, where we choose a plus sign whenever $j$ refers to a vertex in $V_a$ and a minus sign otherwise. Also, $\Vert f \Vert^2_{\mathbb{R}^n} = \mathcal{S}(\Gamma)$, and abbreviating $d_j:=\deg(j)$ we obtain
	\begin{equation*}
	1+\frac{1}{\mathcal{S}(\Gamma)}\left(2\sum_{(i,j) \in V_a \times V_{b}}\frac{a_{ij}}{d_id_j}-\sum_{(i,j) \in V_a \times V_a}\frac{a_{ij}}{d_id_j}-\sum_{(i,j) \in V_{b} \times V_{b}}\frac{a_{ij}}{d_id_j}\right)+\frac{\mathcal{S}_U(\Gamma,2)}{\mathcal{S}(\Gamma)} \leq \lambda_n(U)\ ,
	\end{equation*}
	and from this the statement follows readily.
\end{proof}

\begin{remark}[Optimality of the lower bound for bipartite graphs]
Whenever the graph is bipartite with corresponding disjoint vertex sets $V_a$ and $V_b$, Theorem~\ref{LowerBound} gives the bound 
	\begin{equation*} \left(1+\frac{2R_{-1}(\Gamma,V_a\times V_b)}{\mathcal{S}(\Gamma)}+\frac{\mathcal{S}_U(\Gamma,2)}{\mathcal{S}(\Gamma)}\right) \leq \lambda_n(U)\ .
\end{equation*}
This bound is optimal in the sense that, for a bipartite $d$-regular graph $\Gamma_{d,d}$ with $d = \vert V_a \vert = \vert V_b \vert$ and $U=0$, the lower bound is equal to two. 
\begin{figure}[h]
\begin{tikzpicture}[scale=0.35]
      \tikzset{enclosed/.style={draw, circle, inner sep=0pt, minimum size=.1cm, fill=gray}, every loop/.style={}, every fit/.style={ellipse,draw,inner sep=-2pt,text width=1.5cm, line width=1pt}}

      \node[enclosed] (A) at (0,0) {};
      \node[enclosed] (B) at (0,2) {};
      \node[enclosed] (C) at (0,4) {};
      \node[enclosed] (D) at (0,6) {};
      \node[enclosed] (A') at (8,0) {};
      \node[enclosed] (B') at (8,2) {};
      \node[enclosed] (C') at (8,4) {};
      \node[enclosed] (D') at (8,6) {};
      
      \node [gray,fit=(A) (D),label=below:\textcolor{black!75}{$V_a$}] {};
      \node [gray,fit=(A') (D'),label=below:\textcolor{black!75}{$V_b$}] {};
      \draw (A) edge node[above] {} (A') node[midway, above] (edge1) {};
      \draw (A) edge node[above] {} (B') node[midway, above] (edge2) {};
      \draw (A) edge node[above] {} (C') node[midway, above] (edge3) {};
      \draw (A) edge node[above] {} (D') node[midway, above] (edge4) {};
      \draw (B) edge node[above] {} (A') node[midway, above] (edge1) {};
      \draw (B) edge node[above] {} (B') node[midway, above] (edge2) {};
      \draw (B) edge node[above] {} (C') node[midway, above] (edge3) {};
      \draw (B) edge node[above] {} (D') node[midway, above] (edge4) {};
      \draw (C) edge node[above] {} (A') node[midway, above] (edge1) {};
      \draw (C) edge node[above] {} (B') node[midway, above] (edge2) {};
      \draw (C) edge node[above] {} (C') node[midway, above] (edge3) {};
      \draw (C) edge node[above] {} (D') node[midway, above] (edge4) {};
      \draw (D) edge node[above] {} (A') node[midway, above] (edge1) {};
      \draw (D) edge node[above] {} (B') node[midway, above] (edge2) {};
      \draw (D) edge node[above] {} (C') node[midway, above] (edge3) {};
      \draw (D) edge node[above] {} (D') node[midway, above] (edge4) {};
     \end{tikzpicture}
     \caption{The bipartite $4$-regular graph $\Gamma_{4,4}$ with disjoint sets $V_a$,$V_b$.}\label{fig:dregular}
     \end{figure}
     
     Indeed, considering the bipartite $d$-regular graph as in Figure \ref{fig:dregular} (and potential $U=0$) one can observe quite immediately that $2R_{-1}(\Gamma_{d,d}, V_a \times V_b) = \mathcal{S}(\Gamma_{d,d})$.
\end{remark}
Finally, let us comment on a connection to results obtained in~\cite{JostMulas} where lower (as well as upper bounds) on the largest eigenvalue $\lambda_n(0)$ have also been studied in case of the unweighted normalized Laplacian without external potential (i.e.,  $U=0$) and without loops. In their paper, Jost and Mulas introduce a \emph{Cheeger-like} constant $Q(\Gamma)$ defined by
\[
Q(\Gamma) := \max_{\{i,j\} \in E} \bigg(\frac{1}{\deg(i)}+\frac{1}{\deg(j)} \bigg) \ .
\]
Clearly --by definition-- this constant is dominated by the surface area $\mathcal{S}(\Gamma)$. Furthermore, one can immediately observe that
\begin{align}\label{eq:relation-jost-mulas-surface}
\mathcal{S}(\Gamma) = \sum_{i \in V} \frac{1}{\deg(i)} \leq \sum_{\{i,j\} \in E} \bigg(\frac{1}{\deg(i)} + \frac{1}{\deg(j)}\biggr) \leq \vert E \vert Q(\Gamma)\ .
\end{align}
In \cite[Theorem 3]{JostMulas} it is then shown that
\[
Q(\Gamma) \leq \lambda_n(0) \leq Q(\Gamma) \tau(\Gamma)
\]
where $\tau(\Gamma) < 0.54n$ is some constant.
%
Thus, for unweighted loop-free graphs, using~\eqref{eq:relation-jost-mulas-surface} we immediately arrive at the estimate
\[
\frac{\mathcal{S}(\Gamma)}{\vert E \vert} \leq \lambda_n(0) \leq \mathcal{S}(\Gamma)\tau(\Gamma) \ .
\]


%
%
%
%
\subsection{An improved upper bound}\label{AppendixSpielman}
With respect to Theorem~\ref{TheoremUpperBound} we would like to provide an additional result, namely, an upper bound on $\lambda_2(0)$ for planar graphs (without loops) which is inspired by a result of Spielman and Teng~\cite[Theorem~3.3]{SpTe} who discussed a corresponding bound for the unnormalized Laplacian. Indeed, a generalization of the upper bound derived by Spielman and Teng was already given by Pl\"umer in~\cite[Theorem~3.9]{MP21}: more explicitly, regarding the normalized Laplacian, he showed that
\[
\lambda_2(0) \leq \frac{8\max\deg(\Gamma)}{\sum_{i \in V} \deg(i)} \ ,
\]
where $\max\deg(\Gamma) := \max_{i \in V} \deg(i)$ denotes the \emph{maximal degree} of $\Gamma$. For unweighted graphs this immediately gives
\[
\lambda_2(0) \leq \frac{8\max \deg(\Gamma)}{\mathcal{S}(\Gamma)} \ .
\]
The proof of Spielman and Teng is simple and beautiful and it relies on some geometrical construction which is possible due to planarity of the underlying graph (more explicitly, they employ the \emph{Koebe-Andreev-Thurston embedding theorem}). We shall translate their main ideas to the setting of the normalized weighted Laplacian but estimate differently; by this, we are able to derive an upper bound which is, in some cases, slightly better than the one given above.

For a graph $\Gamma$ we introduce the set of \emph{degree levels} $\deg(\Gamma) := {V/ \hspace{-0.05cm}=_{\deg}}$ as the set of equivalence classes via the relation $=_\mathrm{deg}$ on $V$ which is given by
\[
i =_{\deg} j \:\: :\Longleftrightarrow \:\: \deg(i)=\deg(j) \ . 
\]
Moreover, for $i \in V$, we write $[i]_\mathrm{deg}$ for the corresponding equivalence class with respect to $=_{\deg}$. This allows us to define
\begin{equation*}
\delta(\Gamma):=\sum_{[i]_{\deg} \in \deg(\Gamma)}\frac{1}{\deg(i)}\ .
\end{equation*}
%
We also introduce the quantity
\begin{equation*}
\Theta(\Gamma):=\sum_{i,j: \deg(i)\neq \deg(j) }a_{ij}\left(\frac{1}{\deg(i)^2}+\frac{1}{\deg(j)^2}\right)\  ,
\end{equation*}
and establish the following statement.
\begin{theorem}\label{thm:tweak-spielman} For a finite loop-free planar graph $\Gamma$ one has 
	\begin{equation*}
	\lambda_2(0) \leq \frac{8\delta(\Gamma)+ \Theta(\Gamma)}{\mathcal{S}(\Gamma)}\ .
	\end{equation*}
\end{theorem}
\begin{proof} In a first step, we translate the \emph{Embedding lemma} \cite[Lemma~3.1]{SpTe} to our setting: we obtain
	\begin{equation}\label{Condition}
	\lambda_2(0) = \inf \frac{\frac{1}{2}\sum\limits_{i,j \in V}a_{ij}\Big\|\frac{v_i}{\sqrt{\deg(i)}}-\frac{v_j}{\sqrt{\deg(j)}}\Big\|^2_{\mathbb{R}^n}}{\sum\limits_{i \in V}\|v_i\|^2_{\mathbb{R}^n}}
	\end{equation}
	where the infimum is taken over all vectors $v_1,\dots,v_n \in \mathbb{R}^n$ which are such that 
	\begin{equation}\label{ConditionII}
	\sum_{i=1}^{n}\sqrt{\deg(i)}v_i=0\ .
	\end{equation}
	In a next step, we translate the methods used in the proof of \cite[Theorem~3.3]{SpTe}: The key idea is that the assumed planarity of the graph allows one to embed the graph on a sphere in such a way that two spherical caps (one associated to each vertex) touch each other if and only if two vertices are connected. More explicitly, we use the normalized vectors $w_1,\dots,w_n\in \mathbb{R}^n$ characterized in \cite[Theorem~3.3]{SpTe} and which point to the center of caps on the unit sphere; note that all $w_i$ sum up to zero by construction. We rescale each vector $w_j$ by the factor $\frac{1}{\sqrt{\deg(j)}}$ (in this way we make sure that condition~\eqref{ConditionII} is satisfied), in other words, we replace the vector $v_i$ in \eqref{ConditionII} by $\frac{w_i}{\sqrt{\deg(i)}} \in \mathbb{R}^n$. 
	
	
	With~\eqref{Condition} we then arrive at
	\begin{equation*}\begin{split}
	2\lambda_2(0) &\leq \frac{\sum_{i,j \in V}a_{ij}\big\|\frac{w_i}{\deg(i)}-\frac{w_j}{\deg(j)}\big\|^2_{\mathbb{R}^n}}{\mathcal{S}(\Gamma)} \\
	&\leq  \frac{\sum_{[k]_{\deg} \in \deg(\Gamma)}\frac{1}{\deg(k)^2}\sum_{i,j \in [k]_{\deg}}a_{ij}\left\|w_i-w_j\right\|^2_{\mathbb{R}^n}+2 \Theta(\Gamma)}{\mathcal{S}(\Gamma)} \\
	&\leq  \frac{16\sum_{[k]_{\deg} \in \deg(\Gamma)}\frac{1}{\deg(k)}+2 \Theta(\Gamma)}{\mathcal{S}(\Gamma)} = \frac{16 \delta(\Gamma) + 2 \Theta(\Gamma)}{\mathcal{S}(\Gamma)}\ .
	\end{split}
	\end{equation*}
	Here, the last inequality follows using an estimate similar as in~\cite[Theorem~3.3]{SpTe}: more precisely, for every $k \in V$, one has
	\begin{equation*}\begin{split}
	\sum_{i,j \in [k]_{\deg}}a_{ij}\left\|w_i-w_j\right\|^2_{\mathbb{R}^n} &\leq 2 \sum_{i,j \in [k]_{\deg}}a_{ij} \cdot (r^2_i + r^2_j)\\
	&\leq 4\deg(k) \cdot \sum_{i \in [k]_{\deg}}  r^2_i  \\
	&\leq 16 \deg(k)\ ,
	\end{split}
	\end{equation*}
	where $r_i$ refers to the radius of the cap on the unit sphere associated with the vertex $i$ such that $\Vert w_i - w_j\Vert_{\RR^n}^2$ can be estimated from above by $2(r_i^2 + r_j^2)$ for every pair of vertices $i,j \in V$. Taking into account that $\sum_{i=1}^{n}r^2_i \leq 4$ yields the claimed upper bound.
\end{proof}
Finally, we want to construct a class of graphs where Theorem \ref{thm:tweak-spielman} yields a slightly better upper bound than \cite[Theorem~3.9]{MP21}.
\begin{example}
	For $m,n \in \mathbb{N}$ let $S_{m,n} = S_{m,n}(V,E)$ denote the unweighted graph which is given by a star on $m > 2$ outer vertices which is connected at the central vertex to a path graph on $n$ edges (see Figure~\ref{fig:star-path-graph}) so that $S_{n,m}$ has $\vert V \vert = m+n+1$ vertices.
	\begin{figure}[h]
		\begin{tikzpicture}[scale=0.45]
		\tikzset{enclosed/.style={draw, circle, inner sep=0pt, minimum size=.1cm, fill=gray}, every loop/.style={}}
		
		\node[enclosed, fill=black] (Z) at (0,4) {};
		\node[enclosed, fill=white] (A) at (-1.75,5.75) {};
		\node[enclosed, fill=white] (B) at (-2.5,4) {};
		\node[enclosed, fill=white] (C) at (1.75,5.75) {};
		\node[enclosed, fill=white] (D) at (-1.75,2.25) {};
		\node[enclosed, fill=white] (E) at (1.75,2.25) {};
		\node[enclosed, fill=white] (F) at (0,6.5) {};
		\node[enclosed, fill=white] (G) at (0,1.5) {};
		\node[enclosed] (H) at (2.5,4) {};
		\node[enclosed] (I) at (5,4) {};
		\node[enclosed] (J) at (7.5,4) {};
		\node[enclosed] (K) at (10,4) {};
		\node[enclosed] (L) at (12.5,4) {};

		\draw (Z) edge node[above] {} (A) node[midway, above] (edge1) {};
		\draw (Z) edge node[above] {} (B) node[midway, above] (edge2) {};
		\draw (Z) edge node[above] {} (C) node[midway, above] (edge3) {};
		\draw (Z) edge node[above] {} (D) node[midway, above] (edge4) {};
		\draw (Z) edge node[above] {} (E) node[midway, above] (edge5) {};
		\draw (Z) edge node[above] {} (F) node[midway, above] (edge6) {};
		\draw (Z) edge node[above] {} (G) node[midway, above] (edge7) {};
		\draw (Z) edge node[above] {} (H) node[midway, above] (edge7) {};
		\draw (H) edge node[above] {} (I) node[midway, above] (edge7) {};
		\draw (I) edge node[above] {} (J) node[midway, above] (edge7) {};
		\draw (J) edge node[above] {} (K) node[midway, above] (edge7) {};
		\draw (K) edge node[above] {} (L) node[midway, above] (edge7) {};
		\end{tikzpicture}
  \vspace{-0.3cm}
		\caption{The graph $S_{7,5}$ on $7$ outer (white) and $5$ path (gray) vertices.}\label{fig:star-path-graph}
	\end{figure}
	We first observe that $\delta(S_{m,n}) = 1+ \frac{1}{m+1}+\frac{1}{2}$. Moreover, one can observe that
	\begin{align*}
	\Theta(S_{m,n}) &= 2m\bigg(\frac{1}{1^2} + \frac{1}{(m+1)^2} \bigg) + 2\bigg(\frac{1}{(m+1)^2} + \frac{1}{2^2} \bigg) + 2\bigg(\frac{1}{1^2} + \frac{1}{2^2} \bigg) \\&= 2m  + 3 + \frac{2}{m+1},
	\end{align*}
	whereas for the surface area one has $\mathcal{S}(S_{m,n}) = m + \frac{1}{m+1} + \frac{n+1}{2}$. This yields the upper bound
	\begin{align*}
	\frac{8\delta(S_{m,n}) + \Theta(S_{m,n})}{\mathcal{S}(S_{m,n})} = \frac{8(\frac{3}{2} + \frac{1}{m+1}) + 2m+3+\frac{2}{m+1}}{m + \frac{1}{m+1} + \frac{n+1}{2}} = \frac{15 + \frac{10}{m+1} + 2m}{m + \frac{1}{m+1} + \frac{n+1}{2}}.
	\end{align*}
	On the other hand, the bound obtained in \cite[Theorem~3.9]{MP21} reads as $\frac{8(m+1)}{2m+2n}$. Hence, we want to study the gap
	\[
	\gamma(m,n) := \frac{15 + \frac{10}{m+1} + 2m}{m + \frac{1}{m+1} + \frac{n+1}{2}} - \frac{8(m+1)}{2m+2n}=\frac{2(-2m^3+7m^2+13mn+13m+23n-6)}{(m+n)(2m^2+mn+3m+n+3)},
	\]
	for $m,n \in \mathbb{N}$ and determine cases where this gap is in fact negative. For example, one may set $n=\alpha m$ with a fixed $\alpha \in \mathbb{N}$. Then, it readily follows that $\gamma(m,\alpha m) < 0$ for all $m$ large enough. In addition, by choosing both $m,\alpha$ large enough, one also has
	\[
	\frac{8\delta(S_{m,\alpha m}) + \Theta(S_{m,\alpha m})}{\mathcal{S}(S_{m,\alpha m})} \leq \frac{3}{1+\frac{\alpha}{2}} < 2\ .
	\]
Since $\lambda_2(0)\leq 2$ holds trivially, we thus obtain a non-trivial bound in this case.
\end{example}

\subsection*{Acknowledgement}{} PB was supported by the Deutsche Forschungsgemeinschaft DFG
(Grant 397230547). This article is based on work from COST Action 18232 MAT-DYN-NET, supported by COST (European Cooperation in Science and Technology), https://
www.cost.eu/actions/CA18232/. We shall thank Delio Mugnolo (Hagen) and Davide Bianchi (Shenzhen) for interesting discussions and Amru Hussein (RPTU) for suggesting the name \textit{social graph}.

\appendix

	\vspace*{0.5cm}
	
	{\small
		\bibliographystyle{amsalpha}
		\bibliography{Literature}}

\end{document}